\documentclass[11pt]{amsart}

\usepackage[a4paper,margin=1in]{geometry}
\usepackage{amsmath,amssymb,amsthm,mathtools}
\usepackage[hidelinks]{hyperref}
\usepackage{microtype}

\title{Spanning-Tree Extremality in $C_4$-Free Graphs}
\author{Andr\'as London}
\date{\today}

\address{Institute of Informatics, University of Szeged, Szeged, Hungary}
\email{london@inf.u-szeged.hu}
\urladdr{https://orcid.org/0000-0003-1957-5368}

\newtheorem{theorem}{Theorem}
\newtheorem{conjecture}{Conjecture}
\newtheorem{proposition}{Proposition}
\newtheorem{corollary}{Corollary}
\newtheorem{lemma}{Lemma}
\theoremstyle{definition}

\begin{document}

\begin{abstract}
We study the maximum number of spanning trees in connected $n$-vertex $C_4$-free graphs. For projective-plane orders $n=q^2+q+1$, we determine the spanning-tree count of every polarity graph and show that polarity graphs with exactly $q+1$ absolute points maximize this count within the polarity family, attaining $n^{(n-3)/2}$ spanning trees. Combined with a stability theorem of He, Ma and Yang \cite{HeMaYangCSIAM23}, this yields the same exact upper bound for all sufficiently dense $C_4$-free graphs in the known polarity stability regime. For arbitrary $C_4$-free graphs at these orders, we derive a global upper bound implying
\[
\log \mathrm{st}(n,C_4)=\frac{n-3}{2}\log n+O(\sqrt n),
\]
where $\mathrm{st}(n,C_4)$ denotes the maximum number of spanning trees over connected $n$-vertex $C_4$-free graphs.

\end{abstract}
\maketitle

\section{Introduction}

For a connected graph $G$, let $\tau(G)$ denote its number of spanning trees. Is is sometimes called its
\emph{complexity}. Extremal problems for $\tau(G)$ have been studied under constraints on order, size,
degree sequence, and within various structural graph classes; see, among others,
Kelmans~\cite{Kelmans80,Kelmans96}, Cheng~\cite{Cheng81}, McKay~\cite{McKay83}, and the survey of
degree-based bounds in~\cite{Bozkurt12}. Recent work continues this direction for specific graph
classes: for example, Voblyi and Kononenko~\cite{VoblyiKononenko26} determine maximizers for cactus,
generalized cactus, block graphs, and bi-block graphs. In the present paper we impose instead a
forbidden-subgraph condition of Tur\'an type.

For a fixed graph $H$, define
\begin{equation}\label{eq:def-st}
 \mathrm{st}(n,H):=\max\{\tau(G): |V(G)|=n,\ G\text{ is simple, connected, and }H\text{-free}\}.
\end{equation}
There is no loss in requiring connectivity, since a disconnected graph has no spanning tree. If $G$ is
connected and a missing edge can be added without creating $H$, then adding that edge strictly
increases $\tau(G)$. Hence every maximizer in~\eqref{eq:def-st} is edge-maximal within the class of
$H$-free graphs.

We focus on $H=C_4$. Let $q$ be a prime power and put
\[
 n=q^2+q+1, \qquad m_{\max}:=\frac12q(q+1)^2.
\]
The classical Erd\H{o}s--R\'enyi orthogonal polarity graph $ER_q$ is a $C_4$-free graph on $n$
vertices with $m_{\max}$ edges. F\"uredi proved that for every integer $q\ge14$ no $C_4$-free graph
on $q^2+q+1$ vertices has more than $m_{\max}$ edges~\cite{FurediC4}. It is therefore natural to ask
whether the polarity construction also maximizes the global quantity $\tau(G)$.

Our first observation is that the spanning-tree count can be determined not only for $ER_q$, but for
every polarity graph of a projective plane of order $q$. An absolute point is a point incident with its polar line. If $a(P)$ denotes the number of absolute points, the polarity spectrum implies that
\[
 a(P)=q+1+m_P\sqrt q
\]
for some integer $m_P$; Baer's classical lower bound $a(P)\ge q+1$~\cite{Baer46} then gives
$m_P\ge0$. Proposition~\ref{prop:general-polarity} gives
\begin{equation}\label{eq:intro-polarity-formula}
 \tau(P)=n^{(n-3)/2}
 \left(\frac{q+1-\sqrt q}{q+1+\sqrt q}\right)^{m_P/2}.
\end{equation}
Since the base of the final factor lies in $(0,1)$, the polarity graphs with $m_P=0$, equivalently with
exactly $q+1$ absolute points, maximize $\tau$ within the entire polarity family. Moreover, within the family of polarity graphs of a fixed order, the spanning-tree count is strictly increasing with the number of edges.

This interacts directly with the stability theory of the $C_4$ extremal problem. He, Ma and Yang proved
that for every fixed $c\in(0,1)$ and all sufficiently large even $q$, every $C_4$-free graph $G$ on
$q^2+q+1$ vertices satisfying
\[
 e(G)\ge m_{\max}-\frac c2q
\]
is a subgraph of a unique polarity graph of order $q$~\cite[Theorem~3.1]{HeMaYangCSIAM23}.
Combining this theorem with~\eqref{eq:intro-polarity-formula} proves the exact spanning-tree count
bound in this stability regime; see Section~\ref{sec:stability}.

For arbitrary $C_4$-free graphs, our global estimate is as follows.

\begin{theorem}\label{thm:main}
Let $q\ge14$ be a prime power and let $n=q^2+q+1$. Then
\begin{equation}\label{eq:main-bounds}
 n^{(n-3)/2}
 \le \mathrm{st}(n,C_4)
 \le \left(\frac{n}{n-1}\right)^{n-1}q^q(q+1)^{q^2-2}.
\end{equation}
If $U_q$ denotes the upper bound in~\eqref{eq:main-bounds}, then
\begin{equation}\label{eq:gap-expansion}
 \log\frac{U_q}{n^{(n-3)/2}}
 =\frac q2-\frac14+O\!\left(\frac1q\right).
\end{equation}
Consequently,
\begin{equation}\label{eq:main-asymptotic}
 \log \mathrm{st}(n,C_4)=\frac{n-3}{2}\log n+O(\sqrt n).
\end{equation}
\end{theorem}

We think that the lower bound is exact.

\begin{conjecture}\label{conj:main}
Let $q$ be a prime power and $n=q^2+q+1$. Then
\[
 \mathrm{st}(n,C_4)=n^{(n-3)/2}.
\]
Moreover, every maximizer is a polarity graph with exactly $q+1$ absolute points.
\end{conjecture}

The global upper bound in Theorem~\ref{thm:main} uses the classical degree-sensitive estimate of Grone
and Merris~\cite{GroneMerris88}. Its optimizing degree sequence is already exactly the degree sequence
of $ER_q$, yet the estimate remains larger than $\tau(ER_q)$ by a factor $\exp(q/2+O(1))$. Thus the
present degree-product method cannot prove Conjecture~\ref{conj:main}. The remaining gap suggests that
an exact global argument should exploit additional $C_4$-specific information, such as the codegree
restriction $|N(u)\cap N(v)|\le1$.

\section{Polarity graphs and their spanning-tree counts}\label{sec:polarity}

Let $\Pi$ be a projective plane of order $q$ and let $\varphi$ be a polarity of $\Pi$. The associated
\emph{looped polarity graph} $\widehat P=\widehat P(\Pi,\varphi)$ has the points of $\Pi$ as vertices,
with $x$ adjacent to $y$ whenever $x$ lies on the polar line $\varphi(y)$. Thus $x$ has a loop exactly
when it is an absolute point of the polarity. Let $a(P)$ be the number of absolute points, and let $P$
be the simple graph obtained from $\widehat P$ by deleting its loops. The simple polarity graph $P$ is
$C_4$-free; equivalently, two distinct vertices have at most one common neighbor. When $q$ is a prime
power, an orthogonal polarity of the Desarguesian plane $\mathrm{PG}(2,q)$ gives the classical
Erd\H{o}s--R\'enyi orthogonal polarity graph, denoted $ER_q$. We use standard facts on polarity graphs
from~\cite{PengTaitTimmonsEJC,TaitTimmonsAJC15}.

The looped graph $\widehat P$ is $(q+1)$-regular. If $A$ denotes its adjacency matrix and
$n=q^2+q+1$, then
\begin{equation}\label{eq:polarity-square}
 A^2=J+qI.
\end{equation}
Consequently, the adjacency eigenvalues are $q+1$ and $\pm\sqrt q$, with $q+1$ of multiplicity one.
Let $m_+$ and $m_-$ be the multiplicities of $+\sqrt q$ and $-\sqrt q$. Since
$m_++m_-=n-1$ and $\operatorname{tr}A=a(P)$,
\[
 a(P)=q+1+(m_+-m_-)\sqrt q.
\]
Thus there is an integer $m_P:=m_+-m_-$ such that
\begin{equation}\label{eq:Baer}
 a(P)=q+1+m_P\sqrt q.
\end{equation}
Baer's theorem gives the classical lower bound $a(P)\ge q+1$~\cite{Baer46}, and hence $m_P\ge0$.
The two multiplicities are therefore
\begin{equation}\label{eq:polarity-multiplicities}
 m_{+}=\frac{n-1+m_P}{2},
 \qquad
 m_{-}=\frac{n-1-m_P}{2}.
\end{equation}
Since $n-1=q(q+1)$ is even, $m_P=m_+-m_-$ is even; in particular, $m_P/2$ is a nonnegative integer.
Equation~\eqref{eq:polarity-square} also shows directly
that the graph is connected: for distinct vertices $u,v$, the $(u,v)$ entry of $A^2$ equals $1$, so
there is a walk of length two from $u$ to $v$. With the matrix convention used here, a loop contributes
one to the corresponding diagonal entry of $A$ and one to the corresponding entry of the degree matrix
$D$. Hence deleting the loops leaves the Laplacian $L=D-A$ unchanged and therefore does not change the
spanning-tree count.

\begin{proposition}[Spanning trees of a polarity graph]\label{prop:general-polarity}
Let $P$ be the simple graph of a polarity of a projective plane of order $q$, let
$n=q^2+q+1$, and let $m_P$ be defined by~\eqref{eq:Baer}. Then
\begin{equation}\label{eq:general-polarity-tau}
 \tau(P)=n^{(n-3)/2}
 \left(\frac{q+1-\sqrt q}{q+1+\sqrt q}\right)^{m_P/2}.
\end{equation}
Equivalently, in terms of the number $a(P)$ of absolute points,
\[
 \tau(P)=n^{(n-3)/2}
 \left(\frac{q+1-\sqrt q}{q+1+\sqrt q}\right)^{\frac{a(P)-(q+1)}{2\sqrt q}}.
\]
\end{proposition}

\begin{proof}
Because $\widehat P$ is $(q+1)$-regular, its Laplacian is $L=(q+1)I-A$. By
\eqref{eq:polarity-multiplicities}, the nonzero Laplacian eigenvalues are $q+1-\sqrt q$ with
multiplicity $m_+$ and $q+1+\sqrt q$ with multiplicity $m_-$. Kirchhoff's matrix--tree theorem gives
\[
 \tau(P)=\frac1n(q+1-\sqrt q)^{m_+}(q+1+\sqrt q)^{m_-}.
\]
Using
\[
 (q+1-\sqrt q)(q+1+\sqrt q)=(q+1)^2-q=n
\]
and substituting~\eqref{eq:polarity-multiplicities} yields~\eqref{eq:general-polarity-tau}.
\end{proof}

The same parameter also determines the number of edges of the simple polarity graph. Indeed,
\begin{equation}\label{eq:polarity-edges}
 e(P)=\frac{n(q+1)-a(P)}2
 =m_{\max}-\frac{m_P}{2}\sqrt q.
\end{equation}
Combining~\eqref{eq:general-polarity-tau} and~\eqref{eq:polarity-edges}, one may equivalently write
\[
 \tau(P)=n^{(n-3)/2}\rho_q^{(m_{\max}-e(P))/\sqrt q},
 \qquad
 \rho_q:=\frac{q+1-\sqrt q}{q+1+\sqrt q}.
\]
This gives the following structural consequence.

\begin{corollary}[Extremality within the polarity family]\label{cor:polarity-extremal}
Let $P$ be a polarity graph of a projective plane of order $q$ and let $n=q^2+q+1$. Then
\[
 \tau(P)\le n^{(n-3)/2}.
\]
Equality holds if and only if $m_P=0$, equivalently if and only if $a(P)=q+1$. In particular, for every
prime power $q$ the Erd\H{o}s--R\'enyi orthogonal polarity graph $ER_q$ satisfies
\[
 \tau(ER_q)=n^{(n-3)/2}.
\]
Moreover, for polarity graphs of a fixed order $q$, the spanning-tree count is strictly increasing with
the number of edges: if $e(P_1)>e(P_2)$, then $\tau(P_1)>\tau(P_2)$.
\end{corollary}

\begin{proof}
Put
\[
 \rho_q:=\frac{q+1-\sqrt q}{q+1+\sqrt q}\in(0,1).
\]
By Proposition~\ref{prop:general-polarity}, $\tau(P)=n^{(n-3)/2}\rho_q^{m_P/2}$, and the
preceding discussion gives $m_P\ge0$. This proves the first assertion and its equality condition. Finally,
\eqref{eq:polarity-edges} shows that $e(P)$ is strictly decreasing in $m_P$, while
\eqref{eq:general-polarity-tau} shows that $\tau(P)$ is also strictly decreasing in $m_P$.
\end{proof}

\section{Exact extremality in the stability window}\label{sec:stability}

The previous corollary becomes a statement about arbitrary $C_4$-free graphs once combined with the
stability theorem of He, Ma and Yang. In the form most convenient here, their result states that for
every fixed $c\in(0,1)$ there exists $q_c$ such that, for every even integer $q\ge q_c$, every
$C_4$-free graph $G$ on $q^2+q+1$ vertices satisfying
\begin{equation}\label{eq:HMY-window}
 e(G)\ge m_{\max}-\frac c2q
\end{equation}
is a subgraph of a unique polarity graph of order $q$; see
\cite[Theorem~3.1]{HeMaYangCSIAM23}. Since $P$ also has $q^2+q+1$ vertices, $G$ is necessarily a spanning subgraph of $P$.

\begin{theorem}[Spanning-tree extremality in the stability window]\label{thm:stability-tau}
Fix $c\in(0,1)$. For all sufficiently large even integers $q$, let
$n=q^2+q+1$ and $m_{\max}=\frac12q(q+1)^2$. If $G$ is an $n$-vertex $C_4$-free graph satisfying
\[
 e(G)\ge m_{\max}-\frac c2q,
\]
then
\[
 \tau(G)\le n^{(n-3)/2}.
\]
Equality holds if and only if $G$ is a polarity graph with exactly $q+1$ absolute points. In particular, for even prime powers $q$ equality is attained by $ER_q$.
\end{theorem}

\begin{proof}
If $G$ is disconnected, then $\tau(G)=0$ and there is nothing to prove. By the stability theorem, $G$ is a subgraph of a polarity graph $P$ of order $q$. Since $G$ and $P$ have the same number of vertices, $G$ is a spanning subgraph of $P$. Hence
\[
 \tau(G)\le \tau(P),
\]
with strict inequality if $G$ is a proper spanning subgraph of $P$, since adding an edge to a connected
graph strictly increases its number of spanning trees. Corollary~\ref{cor:polarity-extremal} now gives
\[
 \tau(G)\le \tau(P)\le n^{(n-3)/2}.
\]
Equality therefore requires $G=P$ and $m_P=0$, and these conditions are sufficient.
\end{proof}

\begin{corollary}[Reduction of the global conjecture]\label{cor:counterexample-window}
Fix $c\in(0,1)$. For all sufficiently large even prime powers $q$, any counterexample $G$ to
Conjecture~\ref{conj:main} must satisfy
\[
 e(G)<m_{\max}-\frac c2q.
\]
\end{corollary}

Thus the exact problem is already settled throughout every fixed linear stability window with coefficient
strictly below $1/2$. What remains is the range farther from the extremal edge count, where containment
in a polarity graph is no longer supplied by the known stability theorem.

\section{A global degree-product bound}\label{sec:global}

For the global problem we use the following classical theorem of Grone and Merris~\cite{GroneMerris88};
see also~\cite{Bozkurt12} for a comparison with other degree- and edge-based spanning-tree bounds.

\begin{theorem}[Grone--Merris]\label{thm:GM}
Let $G$ be a connected simple graph on $n\ge2$ vertices, with degrees $d_1,\dots,d_n$ and
$S=\sum_i d_i=2e(G)$. Then
\begin{equation}\label{eq:GM}
 \tau(G)\le
 \left(\frac{n}{n-1}\right)^{n-1}\frac{\prod_{i=1}^n d_i}{S}.
\end{equation}
\end{theorem}

\begin{lemma}[Degree balancing]\label{lem:balance}
For integers $S\ge n\ge2$, let
\[
 B_n(S):=\max\left\{\prod_{i=1}^n d_i:
 d_i\in\mathbb Z_{\ge1},\ \sum_{i=1}^n d_i=S\right\}.
\]
Writing $S=an+r$ with $0\le r<n$, one has
\begin{equation}\label{eq:balanced-product}
 B_n(S)=a^{n-r}(a+1)^r.
\end{equation}
Moreover, $B_n(S)/S$ is strictly increasing in $S$.
\end{lemma}

\begin{proof}
If two entries satisfy $x\ge y+2$, replacing $(x,y)$ by $(x-1,y+1)$ preserves their sum and increases
their product because
\[
 (x-1)(y+1)-xy=x-y-1>0.
\]
Iterating gives~\eqref{eq:balanced-product}. If $S=an+r$, then increasing $S$ by one multiplies the
balanced product by $(a+1)/a$, including the case $r=n-1$. Therefore
\[
 \frac{B_n(S+1)/(S+1)}{B_n(S)/S}
 =\frac{a+1}{a}\frac{S}{S+1}>1,
\]
since $S>a$.
\end{proof}

We recall F\"uredi's sharp edge theorem in the form needed here.

\begin{theorem}[F\"uredi~\cite{FurediC4}]\label{thm:Furedi}
Let $q\ge14$ be an integer and $n=q^2+q+1$. Every $n$-vertex $C_4$-free graph satisfies
\[
 e(G)\le \frac12q(q+1)^2.
\]
For prime powers $q$, equality is attained by polarity graphs with $q+1$ absolute points.
\end{theorem}

\begin{proposition}\label{prop:upper}
Let $q\ge14$, $n=q^2+q+1$, and let $G$ be a connected $C_4$-free graph on $n$ vertices. Then
\begin{equation}\label{eq:upper}
 \tau(G)\le
 \left(\frac{n}{n-1}\right)^{n-1}q^q(q+1)^{q^2-2}.
\end{equation}
\end{proposition}

\begin{proof}
Put $S=\sum_v d(v)=2e(G)$. By Theorem~\ref{thm:Furedi},
\[
 S\le S_{\max}:=q(q+1)^2.
\]
By Theorem~\ref{thm:GM} and Lemma~\ref{lem:balance},
\[
 \tau(G)\le
 \left(\frac{n}{n-1}\right)^{n-1}\frac{B_n(S)}{S}
 \le
 \left(\frac{n}{n-1}\right)^{n-1}\frac{B_n(S_{\max})}{S_{\max}}.
\]
Since
\[
 S_{\max}=q(q+1)^2=qn+q^2,
\]
the balanced sequence with sum $S_{\max}$ has $q+1$ entries equal to $q$ and $q^2$ entries equal to
$q+1$. Hence
\[
 B_n(S_{\max})=q^{q+1}(q+1)^{q^2}.
\]
Dividing by $S_{\max}=q(q+1)^2$ gives~\eqref{eq:upper}.
\end{proof}

\begin{proof}[Proof of Theorem~\ref{thm:main}]
The lower bound follows from Corollary~\ref{cor:polarity-extremal}, while the upper bound is
Proposition~\ref{prop:upper}. Let
\[
 U_q:=\left(\frac{n}{n-1}\right)^{n-1}q^q(q+1)^{q^2-2}.
\]
Then
\[
 \log\frac{U_q}{n^{(n-3)/2}}
 =(n-1)\log\frac{n}{n-1}+q\log q+(q^2-2)\log(q+1)
 -\frac{n-3}{2}\log n.
\]
Substituting $n=q^2+q+1$ and expanding $\log(1+x)$ at $x=0$ gives
\[
 \log\frac{U_q}{n^{(n-3)/2}}
 =\frac q2-\frac14+O\!\left(\frac1q\right),
\]
which is~\eqref{eq:gap-expansion}. Since $q=\sqrt n+O(1)$, equation~\eqref{eq:main-asymptotic}
follows.
\end{proof}

The balancing argument is notable in one respect: the optimizing degree sequence is exactly the degree
sequence of the simple orthogonal polarity graph, namely $q+1$ vertices of degree $q$ and $q^2$
vertices of degree $q+1$. Nevertheless, the Grone--Merris estimate still leaves an a gap exponential in $q$. The next section makes this limitation quantitative.

\section{Edge deficit and the structural gap}\label{sec:deficit}

Let $S_{\max}=2m_{\max}=q(q+1)^2$ and let $U_q$ denote the upper bound in~\eqref{eq:upper}.

\begin{proposition}[Deficit bound]\label{prop:deficit}
Let $q\ge14$, $n=q^2+q+1$, and let $G$ be a connected $C_4$-free graph on $n$ vertices with
\[
 e(G)=m_{\max}-t,
 \qquad 0\le t\le \frac{q^2}{2}.
\]
Then
\begin{equation}\label{eq:deficit}
 \tau(G)\le
 U_q\left(\frac{q}{q+1}\right)^{2t}
 \frac{S_{\max}}{S_{\max}-2t}.
\end{equation}
\end{proposition}

\begin{proof}
Here $S=2e(G)=S_{\max}-2t = q(q+1)^2-2t$. Since $t\le q^2/2$,
\[
 qn\le S\le S_{\max}<(q+1)n.
\]
Writing $S=qn+r =qn+(q^2-2t)$, we have $r=q^2-2t$, and Lemma~\ref{lem:balance} gives
\[
 B_n(S)=q^{q+1+2t}(q+1)^{q^2-2t}
 =B_n(S_{\max})\left(\frac{q}{q+1}\right)^{2t}.
\]
Applying Theorem~\ref{thm:GM} and comparing with $U_q$ yields~\eqref{eq:deficit}.
\end{proof}

Uniformly for $0\le t\le q^2/2$,
\begin{equation}\label{eq:deficit-asymptotic}
 \log\left[
 \left(\frac{q}{q+1}\right)^{2t}
 \frac{S_{\max}}{S_{\max}-2t}
 \right]
 =-\frac{2t}{q}+O\!\left(\frac{t}{q^2}\right).
\end{equation}
On the other hand, Theorem~\ref{thm:main} gives
\[
 \log\frac{U_q}{\tau(ER_q)}=\frac q2-\frac14+O\!\left(\frac1q\right).
\]
Thus the loss in the degree envelope becomes comparable with its intrinsic gap to $ER_q$ only for
$t$ of order $q^2$. This is far outside the order-$q$ regime in which the structural stability theorem
of He, Ma and Yang already forces containment in a polarity graph. The two arguments are therefore
complementary: degree information gives a global asymptotic envelope, while polarity stability gives an
exact result very close to the extremal edge count.

To bridge the region between them, a refinement of the present degree-product approach should exploit
information not used by the degree envelope. Every $C_4$-free graph satisfies
\begin{equation}\label{eq:codegree}
 |N(u)\cap N(v)|\le1\quad(u\ne v),
 \qquad
 \sum_{v\in V(G)}\binom{d(v)}2\le\binom n2.
\end{equation}
Neither constraint is used by Proposition~\ref{prop:upper} after the edge bound is inserted. A natural
next step toward Conjecture~\ref{conj:main} is therefore a determinant or Laplacian inequality that
incorporates the codegree restriction~\eqref{eq:codegree}. Quantitatively, a refinement of the present
degree envelope would need to recover an additional factor of order $\exp(-q/2)$.

\section{Concluding remarks}

The definition~\eqref{eq:def-st} makes sense for any forbidden graph $H$. Applying the
arithmetic--geometric mean inequality to the Grone--Merris estimate recovers Grimmett's classical
edge-only bound~\cite{Grimmett76} and yields
\begin{equation}\label{eq:general-H}
 \mathrm{st}(n,H)
 \le \frac1n\left(\frac{2\,\mathrm{ex}(n,H)}{n-1}\right)^{n-1}.
\end{equation}
For fixed $t\ge2$, F\"uredi's asymptotic theorem for $K_{2,t}$~\cite{FurediK2t} consequently gives
\[
 \log \mathrm{st}(n,K_{2,t})\le \frac n2\log n+O_t(n),
\]
while the Bondy--Simonovits theorem~\cite{BondySimonovits74} gives, for fixed $k\ge2$,
\[
 \log \mathrm{st}(n,C_{2k})\le \frac nk\log n+O_k(n).
\]
These observations suggest a general question: when does an edge-extremal construction for a forbidden
subgraph also maximize the number of spanning trees? For $C_4$, Theorem~\ref{thm:stability-tau}
answers this affirmatively in every fixed stability window
$e(G)\ge m_{\max}-\frac c2q$ with $c<1$ and $q$ sufficiently large and even, while the global problem
and the behavior of $\mathrm{st}(n,C_4)$ away from projective-plane orders remain open.\\

\noindent \textbf{AI tools disclosure.} Generative AI tools were used for language polishing and minor editorial suggestions.

\bibliographystyle{plain}

\end{document}